\documentclass[11pt,a4paper]{amsart}
\usepackage[utf8]{inputenc}
\usepackage[a4paper,margin=2.5cm]{geometry}
\usepackage{amsmath,amssymb,amsthm,mathtools,abraces}
\usepackage{stmaryrd}
\usepackage{color,xcolor}
\usepackage[hypertexnames=false,hyperfootnotes=false,colorlinks=true,linkcolor=blue,%
citecolor=purple,filecolor=magenta,urlcolor=cyan,unicode,linktocpage=true,pagebackref=true]{hyperref}
\usepackage{yhmath}
\usepackage{verbatim}

\newcommand\be{\begin{equation}}
\newcommand\ee{\end{equation}}
\newcommand\nn{\nonumber}

\newcommand\p{\partial}
\DeclareMathOperator{\odd}{odd}
\DeclareMathOperator{\even}{even}
\DeclareMathOperator{\res}{Res}

\newcommand\normordboson{ {\scriptstyle {{*}\atop{*}}} }

\newcommand{\<}{\left <}
\renewcommand{\>}{\right >}
\DeclareMathOperator{\gl}{\mathfrak{gl}}

\makeatletter
\@namedef{subjclassname@2020}{%
  \textup{2020} Mathematics Subject Classification}
\makeatother

\newtheorem{theorem}{Theorem}
\newtheorem{lemma}{Lemma}[section]
\newtheorem{proposition}[lemma]{Proposition}
\newtheorem{corollary}[lemma]{Corollary}
\newtheorem{remark}{Remark}[section]

\newtheorem*{theorem*}{Theorem}

\numberwithin{equation}{section}

\usepackage{filecontents}

\title[Cut-and-join operators for higher Weil--Petersson volumes]{Cut-and-join operators for higher Weil--Petersson volumes}

\author{Alexander Alexandrov}

\address{Center for Geometry and Physics, Institute for Basic Science (IBS), Pohang 37673, Korea
}

\email{ {\tt alexandrovsash at gmail.com}}


\date{\today}

\begin{document}

\begin{abstract} 
In this paper, we construct the cut-and-join operator description for the generating functions of all intersection numbers of $\psi$, $\kappa$, and $\Theta$ classes on the moduli spaces  $\overline{\mathcal M}_{g,n}$. The cut-and-join operators define an algebraic version of topological recursion. This recursion allows us to compute all these intersection numbers recursively. For the specific values of parameters, the generating functions describe the volumes of moduli spaces of (super) hyperbolic Riemann surfaces with geodesic boundaries, which are also related to the Jackiw--Teitelboim (JT) (super)gravity.
\end{abstract}

\maketitle


\def\thefootnote{\arabic{footnote}}

\section{Introduction}
\addcontentsline{toc}{section}{Introduction}
\setcounter{equation}{0}

Let $\overline{\mathcal M}_{g,n}$ be the Deligne--Mumford compactification of the moduli space of stable complex curves of genus $g$ with $n$ distinct marked points. The moduli space $\overline{\mathcal M}_{g,n}$ is defined to be empty unless the stability condition
\be
2g - 2 + n > 0
\ee
is satisfied. Let us associate with a marked point a line bundle ${\mathbb{L}}_i$ whose fiber at a moduli point $(\Sigma;z_1,\dots,z_n)$ is the complex cotangent line $T^*_{z_i}$ to $\Sigma$ at $z_i$.   Let $\psi_i\in H^2(\overline{\mathcal{M}}_{g,n},\mathbb{Q})$ denote the first Chern class of ${\mathbb{L}}_i$. With the forgetful map $\pi: \overline{\mathcal M}_{g,n+1} \rightarrow \overline{\mathcal M}_{g,n} $ we define the {\em Miller--Morita--Mumford tautological classes} \cite{Mumford}, $\kappa_k:= \pi_* \psi_{n+1}^{k+1} \in H^{2k}(\overline{\mathcal{M}}_{g,n},\mathbb{Q})$. Following Norbury \cite{Norb}, we introduce $\Theta$ classes, $\Theta_{g,n}\in H^{4g-4+2n}(\overline{\mathcal{M}}_{g,n})$, which are also related to the super Riemann surfaces \cite{NorbS}. We refer the reader to \cite{Norb,NP1,NorbS} for a detailed description of $\Theta$ classes. For $\alpha\in \{0,1\}$  we consider the intersection numbers
\be\label{gin}
\<\kappa_{b_1}\kappa_{b_2}\cdots \kappa_{b_m}\tau_{a_1}\tau_{a_2}\cdots\tau_{a_n} \>_g^{(\alpha)}:= \int_{\overline{\mathcal{M}}_{g,n}} \Theta_{g,n}^{1-\alpha}\kappa_{b_1}\kappa_{b_2}\dots \kappa_{b_m}\psi_1^{a_1} \psi_2^{a_2}\dots \psi_n^{a_n} \in {\mathbb Q}.
\ee
We do not consider $\kappa_0$ class in this paper, therefore below we assume $b_j>0$. We call such intersection numbers the {\em higher Weil--Petersson volumes}.
These intersection numbers vanish, unless the corresponding complex dimensions coincide
\be\label{dimc}
\sum_{i=1}^n a_i + \sum_{j=1}^{m}b_j+(1-\alpha)(2g-2+n)=\dim_{\mathbb C}\,\overline{\mathcal M}_{g,n},
\ee
where $\dim_{\mathbb C}\,\overline{\mathcal M}_{g,n}=3g-3+n$. Because of the dimensional constraints $\Theta_{g,n}^2=0$, so \eqref{gin} are the most general intersection numbers of $\psi$, $\kappa$, and $\Theta$ classes.  In this paper we derive an explicit formula for the generating functions of all these intersections numbers for $\alpha \in \{0,1\}$. This formula allows us to compute all higher Weil--Petersson volumes recursively.

Let $T_k$, $k\geq 0$ be formal variables. Consider the generating functions of the intersection numbers \eqref{gin} without $\kappa$ classes,
\be\label{taua}
\tau_{\alpha}=\exp\left(\sum_{g=0}^\infty \sum_{n=1}^\infty \hbar^{2g-2+n}F_{g,n}^\alpha\right),
\ee
where
\be
F_{g,n}^\alpha=\sum_{a_1,\ldots,a_n\geq 0}\<\tau_{a_1}\tau_{a_2}\cdots\tau_{a_n}\>_g^{(\alpha)}\frac{\prod T_{a_i}}{n!}.
\ee
On substitution $T_k=(2k+1)!!\, t_{2k+1}$ these generating functions yield the profound Kontsevich--Witten tau-function  \cite{Kon,Wit} (for $\alpha=1$) and the 
Br\'ezin--Gross--Witten tau-function  \cite{Norb} (for $\alpha=0$).  These tau-functions are solutions of the KdV hierarchy. We will apply the methods of integrable hierarchies below, hence we work with the variables $t_{k}$, which are natural for these hierarchies.

Let us consider operators 
\begin{equation}\label{CAJ}
\begin{split}
\widehat{W}_0&=\sum_{k,m\in {\mathbb Z}_{\odd}^+} \left(kmt_{k}t_{m}\frac{\p}{\p t_{k+m-1}}+\frac{1}{2}(k+m+1)t_{k+m+1}\frac{\p^2}{\p t_k \p t_m}\right)+\frac{t_1}{8},\\
\widehat{W}_1&=\frac{1}{3}\sum_{k,m\in {\mathbb Z}_{\odd}^+} \left(kmt_{k}t_{m}\frac{\p}{\p t_{k+m-3}}+\frac{1}{2}(k+m+3)t_{k+m+3}\frac{\p^2}{\p t_k \p t_m}\right)+\frac{t_1^3}{3!}+\frac{t_3}{8}.
\end{split}
\end{equation}
From the Virasoro constraints, satisfied by $\tau_{\alpha}$, it follows 
\begin{theorem}[\cite{KSS,KS2}]\label{TW}
\be\label{Wa}
\tau_{\alpha}=\exp\left({\hbar \widehat{W}_{\alpha}}\right)\cdot 1.
\ee
\end{theorem}
We call such operators the {\em cut-and-join operators}, see \cite[Section 2]{H3} for more details.

 Let us consider the generating functions of the higher Weil--Petersson volumes \eqref{gin}
 \be
\tau_\alpha({\bf t}, {\bf s})=\left.\exp\left(\sum_{g=0}^\infty \sum_{n=0}^\infty \hbar^{2g-2+n}F_{g,n}^\alpha({\bf T}, {\bf s})\right)\right|_{T_k=(2k+1)!! t_{2k+1}},
\ee
where
 \be\label{FE}
F_{g,n}^\alpha({\bf T}, {\bf s})=\sum_{a_1,\ldots,a_n\ge 0}\sum_{b_1,\ldots,b_m>  0}\<\kappa_{b_1}\kappa_{b_2}\cdots \kappa_{b_m}\tau_{a_1}\tau_{a_2}\cdots\tau_{a_n} \>_g^{(\alpha)} \frac{\prod T_{a_i}}{n!}  \frac{\prod s_{b_i}}{m!}.
\ee
These generating functions contain all possible intersection numbers of $\psi_k$, $\Theta_{g,n}$, and $\kappa_b$ classes for $b>0$.

According to Manin and Zograf \cite{MZ}, insertion of $\kappa$ classes into the intersection numbers of $\psi$ classes can be described by certain translation of the variables ${\bf t}$ in the tau-function $\tau_1({\bf t})$. For the $\Theta$-case an analogous statement was recently obtained by Norbury \cite{NorbS}.
To describe these we introduce the polynomials $q_j({\bf s})$, defined by the generating function
\be\label{genss}
1-\exp\left(-\sum_{j=1}^\infty s_j z^j\right)=\sum_{j=1}^\infty q_j({\bf s})z^j,
\ee
and we put $q_k({\bf s})=0$ for $k<1$.
For $k>0$ they are nothing but the negative of the {\em elementary Schur functions} 
\be
q_j({\bf s})=-p_j(-{\bf s}),
\ee
where
\be
\exp\left(\sum_{j=1}^\infty s_j z^j\right)=:\sum_{j=0}^\infty p_j({\bf s})z^j.
\ee
Then the generating functions $\tau_\alpha({\bf t}, {\bf s})$ can be obtained from the generating functions $\tau_\alpha({\bf t})$ without $\kappa$-classes by translation
\begin{theorem}[\cite{MZ,NorbS}]\label{TMZN}
For $\alpha\in\{0,1\}$,
\be\label{tautr}
\tau_\alpha({\bf t}, {\bf s})=\tau_\alpha\left(\left\{ t_{2k+1}+\frac{1}{\hbar}\frac{q_{k-\alpha}({\bf s})}{(2k+1)!!}\right\}\right). 
\ee
\end{theorem}
Note that only the variables $t_{2k+1}$ for $k>\alpha$ are translated. This theorem, in particular, shows that $\tau_\alpha({\bf t}, {\bf s})$ are KdV tau-functions in the variables ${\bf t}$ for the arbitrary values of parameters ${\bf s}$.

From this theorem and cut-and-join formula \eqref{Wa}
in this paper we derive the cut-and-join description of the tau-functions $\tau_\alpha({\bf t}, {\bf s})$. Let 
\be\label{fa}
f_\alpha(z)=\left(\frac{2\alpha+1}{2{\bf i} \sqrt{2\pi}}\int_\gamma\frac{{\hbox{d}} t}{t^{\alpha+\frac{3}{2}}}\exp\left(\frac{z^2 t}{2}-\sum_{j=1}^\infty s_j t^{-j}\right)\right)^\frac{1}{2\alpha+1},
\ee
where $\gamma$ is Hankel's loop, which runs from $-\infty$ around the origin counter clockwise and returning to $-\infty$, and ${\bf i}=\sqrt{-1}$, be the formal series, $f_\alpha(z)\in z+z{\mathbb Q}[{\bf s}][\![z^2]\!]$. 
We also introduce the coefficients
\be\label{sigmar}
\rho_\alpha[k,m]= [z^m] f_\alpha(z)^k,
\ee
and the operators
\begin{equation}\label{rrr}
 \begin{aligned}
 \widehat{J}_{k}^{\alpha}&= \sum_{m=k}^\infty \rho_\alpha[k,m] \widehat{J}_{m},                                                     & &k\in {\mathbb Z},\\
\end{aligned}
\end{equation}
where
\be
\widehat{J}_k =
\begin{cases}
\displaystyle{\frac{\p}{\p t_k} \,\,\,\,\,\,\,\,\,\,\,\, \mathrm{for} \quad k>0},\\[2pt]
\displaystyle{0}\,\,\,\,\,\,\,\,\,\,\,\,\,\,\,\,\,\,\, \mathrm{for} \quad k=0,\\[2pt]
\displaystyle{-kt_{-k} \,\,\,\,\,\mathrm{for} \quad k<0.}
\end{cases}
\ee
Consider the operators
\begin{equation}
\begin{split}\label{Wpri}
{\widehat{W}_0}({\bf s})&=\sum_{k,m\in {\mathbb Z}_{\odd}^+}^\infty\left( {\widehat {J}}_{-k}^0  {\widehat {J}}_{-m}^0  {\widehat {J}}_{k+m-1}^0 +\frac{1}{2} {\widehat {J}}^0_{-k-m-1} {\widehat {J}}_{k}^0  {\widehat {J}}_{m}^0\right)+\frac{{\widehat {J}}_{-1}^0}{8},\\
{\widehat{W}_1}({\bf s})&=\frac{1}{3}\sum_{k,m\in {\mathbb Z}_{\odd}^+}^\infty \left(  {\widehat {J}}_{-k}^1  {\widehat {J}}_{-m}^1   {\widehat {J}}_{k+m-3}^1
+\frac{1}{2} {\widehat {J}}_{-k-m-3}^1 {\widehat {J}}_{k}^1 {\widehat {J}}_{m}^1\right)+\frac{ ({\widehat {J}}_{-1}^1)^3}{3!}+\frac{ {\widehat {J}}_{-3}^1}{24}.
\end{split}
\end{equation}
These operators are cubic in ${\widehat{J}}_k$.
The main result of this paper is the following 
\begin{theorem}\label{MT}
\be\label{taum}
\tau_\alpha({\bf t}, {\bf s})=\exp\left(\hbar {\widehat{W}_\alpha}({\bf s})\right)\cdot 1.
\ee
\end{theorem}
This result might be a bit surprising. 
In general, there is no reason to expect that the generating functions $\tau_{\alpha}({\bf t},{\bf s})$ are given by the relations of the form \eqref{taum} with cubic operators ${\widehat{W}_\alpha}({\bf s})$. We believe that this is a general property of a big, still unspecified class of the cohomological field theories. Note that the operators ${\widehat{W}_\alpha}({\bf s})$ are independent of $\hbar$.

Let us consider the {\em topological expansion} of the tau-functions $\tau_{\alpha}({\bf t},{\bf s})$,
\be
\tau_{\alpha}({\bf t},{\bf s})=\sum_{p=0}^\infty \tau_{\alpha}^{(p)}({\bf t},{\bf s}) \hbar^p.
\ee
Coefficients $\tau_{\alpha}^{(p)}({\bf t},{\bf s})$ are homogeneous polynomials of degree $(2\alpha+1)p$ in ${\bf t}$ and ${\bf s}$ if we put $\deg t_k = k$ and $\deg s_k=2k$.
Then an equivalent formulation of Theorem \ref{MT} is given by the following {\em algebraic topological recursion}
\begin{corollary}
\be\label{pp}
\tau_{\alpha}^{(p)}({\bf t},{\bf s})=\frac{1}{p} {\widehat{W}_\alpha} ({\bf s}) \cdot \tau_{\alpha}^{(p-1)}({\bf t},{\bf s}).
\ee
\end{corollary}
This recursion allows us to find recursively all higher Weil--Petersson volumes \eqref{gin} with the initial condition $\tau_{\alpha}^{(0)}({\bf t},{\bf s})=1$.

There are several procedures, which allow to compute the higher Weil--Petersson volumes \eqref{gin} recursively. The Chekhov--Eynard--Orantin topological recursion \cite{EWP}  allows us to find recursively all the intersection numbers \eqref{gin} for $\alpha=1$. It is equivalent to the recursion of  Liu--Xu \cite{LiuXu}. 
Mirzakhani \cite{Mirz1,Mirz2} derived a recursion for the Weil--Petersson volumes of the moduli spaces of hyperbolic Riemann surfaces with geodesic boundaries associated with the case $\alpha=1$ and $s_k=\delta_{k,1}s$. For the super hyperbolic Riemann surfaces with geodesic boundaries, corresponding to the case $\alpha=0$ and $s_k=\delta_{k,1}s$, an analog of Mirzakhani's recursion was recently derived by Stanford and Witten \cite{SW}. The Chekhov--Eynard--Orantin topological recursion for this case is proven by Norbury \cite{NorbS}. We believe that  for the proper specifications of parameters all these recursions should follow from the algebraic topological recursion \eqref{pp}.

In physics, specifications of the intersection numbers of the form  \eqref{gin} are related to the versions of topological gravity known as JT gravity (for $\alpha=1$) and JT supergravity (for $\alpha=0$) \cite{JTMM,OS}.
Physical interpretation of the cut-and-join description \eqref{taum} is not clear yet.

The present paper is organized as follows. In Section \ref{S1} we describe how the insertion of the $\kappa$ classes can be described by the Virasoro group elements. Section \ref{S2} is devoted to construction of the cut-and-join operators. In Appendix \ref{AAA} we present the first few coefficients of the cut-and-join operators ${\widehat{W}_\alpha}({\bf s})$. 

\subsection*{Acknowledgments}

This work was supported by the Institute for Basic Science (IBS-R003-D1).


\section{Virasoro group operator}\label{S1}
Consider the space ${\mathbb C}[\![{\bf t}]\!]$ of formal power series of infinitely many variables ${\bf t}=\{ t_1, t_2,\dots\}$. We denote by $\widehat{\dots}$ the differential operators acting on this space. We introduce the {\em Heisenberg--Virasoro algebra} ${\mathcal L}$. It is generated by the operators
\be
\widehat{J}_k =
\begin{cases}
\displaystyle{\frac{\p}{\p t_k} \,\,\,\,\,\,\,\,\,\,\,\, \mathrm{for} \quad k>0},\\[2pt]
\displaystyle{0}\,\,\,\,\,\,\,\,\,\,\,\,\,\,\,\,\,\,\, \mathrm{for} \quad k=0,\\[2pt]
\displaystyle{-kt_{-k} \,\,\,\,\,\mathrm{for} \quad k<0,}
\end{cases}
\ee
the unit, and the Virasoro operators
\be
\label{virfull}
\widehat{L}_m=\frac{1}{2} \sum_{a+b=-m}a b t_a t_b+ \sum_{k=1}^\infty k t_k \frac{\p}{\p t_{k+m}}+\frac{1}{2} \sum_{a+b=m} \frac{\p^2}{\p t_a \p t_b}.
\ee
These operators satisfy the commutation relations
\begin{align}\label{comre}
\left[\widehat{J}_k,\widehat{J}_m\right]&=k \delta_{k,-m},\nn\\
\left[\widehat{L}_k,\widehat{J}_m\right]&=-m \widehat{J}_{k+m},\\
\left[\widehat{L}_k,\widehat{L}_m\right]&=(k-m)\widehat{L}_{k+m}+\frac{1}{12}\delta_{k,-m}(k^3-k).\nn
\end{align}

This algebra is  a subalgebra of $\gl(\infty)$, acting on the space of solutions of the KP integrable hierarchy.
Let us describe the relation between the action of the translation and Virasoro group operators on the tau-functions $\tau_\alpha$ given by \eqref{taua}.
For a Virasoro group element of the form
\be\label{Virg}
\widehat{V}_{\bf a}=\exp\left({\sum_{k>0} a_k \widehat{L}_k }\right)
\ee
consider the series
\be
\Xi(\widehat{V}_{\bf a})=e^{\sum_{k>0} a_k {\mathtt l}_k } \,z \, e^{-\sum_{k>0} a_k {\mathtt l}_k } \in z+z{\mathbb C}[\![z]\!],
\ee
where
\be\label{virw}
{\mathtt l}_m=-z^m\left(z\frac{\p}{\p z}+\frac{m+1}{2}\right).
\ee
For a given set of parameters $a_k\in {\mathbb C}$, $k\in {\mathbb Z}_{>0}$ we denote the  series  associated with the operator ${\widehat{V}}_{\bf a}$ by $f(z;{\bf a})=\Xi(\widehat{V}_{\bf a})$.
We also introduce the inverse formal series $h(z;{\bf a})$,
\be\label{hfun}
f(h(z;{\bf a});{\bf a})=z.
\ee

For any series $f(z)\in z+z{\mathbb C}[\![z]\!]$ we construct a corresponding element of the Virasoro group 
\be\label{V}
\widehat{V}=\Xi^{-1}(f(z)),
\ee
where the coefficients $a_k$ are defined implicitly by
\be
f(z)=e^{\sum_{k>0} a_k {\mathtt l}_k } \,z \, e^{-\sum_{k>0} a_k {\mathtt l}_k } .
\ee

Tau-functions $\tau_\alpha$ given by \eqref{taua} satisfy the Virasoro constraints
\begin{equation}
\begin{split}\label{VirC}
\widehat{L}_k^{\alpha} \cdot {\tau}_{\alpha} &=0, \,\,\,\, k \geq -\alpha, \\
\end{split}
\end{equation}
where the Virasoro operators are given by
\begin{equation}
\begin{split}
\widehat{L}_k^{\alpha}&=\frac{1}{2}\widehat{L}_{2k}-\frac{1}{2\hbar}\frac{\p}{\p t_{2k+1+2\alpha}} +\frac{\delta_{k,0}}{16} \in {\mathcal L}.
\end{split}
\end{equation} 
These operators satisfy the commutation relations of the Virasoro algebra,
\be
\left[\widehat{L}_k^{\alpha},\widehat{L}_k^{\alpha}\right]=(k-m)\widehat{L}_{k+m}^{\alpha}, \quad \quad k,m \geq -\alpha.
\ee

Using these Virasoro constraints for tau-functions $\tau_\alpha$ we can rewrite the action of the translation operators in terms of the Virasoro group elements. Let the functions $v^\alpha(z;{\bf a})$ and $f(z;{\bf a})$ be related by 
\be\label{vfunc}
{v}^\alpha(z;{\bf a})=\frac{1}{(2\alpha+1)}  (z^{2\alpha+1}-f(z;{\bf a})^{2\alpha+1}) \in z^{2+2\alpha}{\mathbb C}[\![z]\!],
\ee
where  $v^\alpha(z;{\bf a})$ and $f(z;{\bf a})$ are odd functions of $z$.
Consider the coefficients
\be
v_k^{\alpha}= [z^k] {v}^\alpha(z;{\bf a}).
\ee

\begin{lemma}[\cite{H3}]\label{virtos}
We have
\be
\exp\left({\hbar^{-1} \sum_{k=\alpha+1} v_{2k+1}^{\alpha}\frac{\p}{\p t_{2k+1}}}\right) \cdot \tau_\alpha=\exp\left({\sum_{k>0} a_{2k} \widehat{L}_{2k} }\right) \cdot \tau_\alpha.
\ee
\end{lemma}

Let us find the functions $f_\alpha( z)$ associated with the transitions in Theorem \ref{TMZN}. These translations are described by the generating functions
\be
{v}^\alpha(z)=\sum_{k=1}^\infty q_k({\bf s})\frac{z^{2k+2\alpha+1}}{(2k+2\alpha+1)!!}.
\ee
Functions ${v}^\alpha(z)$ satisfy
\be
\frac{1}{z}\frac{\p}{\p z} {v}^1(z) ={v}^0(z).
\ee
\begin{remark}
In the right hand side of \eqref{tautr}
the parameters ${\bf s}$ parametrize arbitrary translations of the variables $t_{2k+1}$ for $k>\alpha$.
\end{remark}
Then $f_\alpha( z)$ are the solutions of \eqref{vfunc},
\be\label{ftov}
f_\alpha(z)=\left(z^{2\alpha+1}+({2\alpha+1})v^\alpha(z)\right)^{\frac{1}{2\alpha+1}}.
\ee
Recall the integral representation for the reciprocal  gamma function
\be
\frac{1}{\Gamma(x)}=\frac{1}{2\pi {\bf i}} \int_{\gamma} e^t t^{-z}  {\hbox{d}} t,
\ee
where $\gamma$ is Hankel's loop and ${\bf i}=\sqrt{-1}$. Using the identities
\be
(2k+1)!!=2^{k+1}\frac{\Gamma(k+3/2)}{\Gamma(1/2)}
\ee
and $\Gamma(1/2)=\sqrt{\pi}$, one has
\begin{equation}
\begin{split}\label{vv}
{v}^\alpha(z)&=\frac{1}{2{\bf i} \sqrt{2\pi}}\int_\gamma\frac{{\hbox{d}} t}{t^{\alpha+\frac{3}{2}}}
e^\frac{z^2 t}{2}\left(1-\exp\left(-\sum_{j=1}^\infty s_j t^{-j}\right)\right)\\
&=\frac{z^{2\alpha+1}}{2\alpha+1}-\frac{1}{2{\bf i} \sqrt{2\pi}}\int_\gamma\frac{{\hbox{d}} t}{t^{\alpha+\frac{3}{2}}}\exp\left(\frac{z^2 t}{2}-\sum_{j=1}^\infty s_j t^{-j}\right),
\end{split}
\end{equation}
and from \eqref{ftov} we have  \eqref{fa}.
The first few terms of the series expansion of $f_\alpha(z)$ are given by
\begin{equation}
 \begin{aligned}
f_0(z)&=z-\frac{1}{3}\,s_{{1}}{z}^{3}-\frac{1}{15}\, \left( s_{{2}}-\frac{1}{2}\,{s_{{1}}}^{2}
 \right) {z}^{5}-\frac{1}{105}\left(s_
{3}+ \frac{1}{6}\,{s_{{1}}}^{3}-s_{{2}}s_{{1}}\right) {z}^{7}\\
&-\frac{1}{945} { \left(s_{{4}} -\frac{1}{24}\,{s_{{1}}}^{4}+\frac{1}{2}\,
s_{{2}}{s_{{1}}}^{2}-s_{{3}}s_{{1}}-\frac{1}{2}\,{s_{{2}}}^{2}
 \right) {z}^{9}}+\dots,\\
f_1(z)&=z-\frac{1}{15}s_1z^3-\frac{1}{105}\left(s_{{2}}-\frac{1}{30}\,{s_{{1}}}^{2}\right)z^5-\frac{1}{945}\left(s_{{3}}+\frac{1}{30}\,{s_{{1}}}^{3}+\frac{1}{5}\,s_{{2}}s_{{1}}\right)z^7\\
&-\frac{1}{10395}\left({s_{{4}}+{\frac {289}{12600}}\,{s_{{1}}}^{4}+\frac {61}{
210}}\,s_{{2}}{s_{{1}}}^{2}+{\frac {7}{15}}\,s_{{3}}s_{{1}}+{\frac {31}{70}}\,{s_{{2}}}^{2}\right)z^9+\dots,
\end{aligned}
\end{equation}
where by $\dots$ we denote the higher order terms. 

Let 
\be\label{Vopp}
{\widehat V}_{\alpha}({\bf s})=\Xi^{-1}(f_\alpha(z)).
\ee
Then from Lemma \ref{virtos} one has
\begin{proposition}\label{propV}
\be
\tau_\alpha({\bf t}, {\bf s})={\widehat V}_{\alpha}({\bf s}) \cdot \tau_\alpha({\bf t}).
\ee
\end{proposition}
Note, that $f_\alpha(z)$ and ${\widehat V}_{\alpha}({\bf s}) $ do not depend on $\hbar$.

\section{Cut-and-join operators}\label{S2}
\subsection{Proof of Theorem \ref{MT}}
For the Virasoro group element \eqref{Virg} we introduce operators
\begin{equation}\label{LJa}
\begin{split}
{\widehat{J}}_{k}^{{\bf a}}=\widehat{V}_{\bf a}  \widehat{J}_{k} \widehat{V}_{\bf a}^{-1}. 
\end{split}
\end{equation}
From the commutation relations of the Heisenberg--Virasoro algebra ${\mathcal L}$ it follows that
these are linear combinations of the operators $\widehat{J}_{m}$.
Let us also introduce the coefficients
\be
\rho[k,m]= [z^m] f(z;{\bf{a}})^k,
\ee
where $f(z,{\bf a})=\Xi(\widehat{V}_{\bf a})$. They are polynomials in parameters ${\bf a}$. Then,
\begin{lemma}[\cite{H3}]\label{lemma32}
\begin{equation}
 \begin{aligned}
 \widehat{J}_{k}^{{\bf a}}&= \sum_{m=k}^\infty \rho[k,m] \widehat{J}_{m},                                                     \quad \quad k\in {\mathbb Z}.\\
\end{aligned}
\end{equation}
\end{lemma}

Let us consider the coefficients $\rho_\alpha[k,m]$ for the functions $f_\alpha(z)$ given by \eqref{fa},
\be
\rho_\alpha[k,m]= [z^m] f_\alpha(z)^k
\ee
and $ \widehat{J}_{k}^{\alpha}= \sum_{m=k}^\infty \rho_\alpha[k,m] \widehat{J}_{m}$. Then 
\be
 \widehat{J}_{k}^{\alpha}= {\widehat V}_{\alpha}({\bf s}) \widehat{J}_{k} {\widehat V}_{\alpha}({\bf s})^{-1}.
\ee
Therefore, the operators \eqref{Wpri} can be represented as
\be
{\widehat{W}_\alpha}({\bf s})= {\widehat V}_{\alpha}({\bf s}) \widehat{W}_{\alpha} {\widehat V}_{\alpha}({\bf s})^{-1},
\ee
where ${\widehat{W}_\alpha}$ are given by \eqref{CAJ}. In particular, we have ${\widehat{W}_\alpha}({\bf 0})={\widehat{W}_\alpha}$.
\begin{proof}[Proof of Theorem \ref{MT}]
From  Theorem \ref{TW}, Lemma \ref{virtos}, and Proposition \ref{propV} we have
\begin{equation}
\begin{split}
\tau_\alpha({\bf t}, {\bf s})&={\widehat V}_{\alpha}({\bf s}) \exp\left({\hbar \widehat{W}_{\alpha}}\right)\cdot 1\\
&={\widehat V}_{\alpha}({\bf s}) \exp\left({\hbar \widehat{W}_{\alpha}}\right) {\widehat V}_{\alpha}({\bf s})^{-1}\cdot 1\\
&= \exp\left(\hbar {\widehat V}_{\alpha}({\bf s}) \widehat{W}_{\alpha} {\widehat V}_{\alpha}({\bf s})^{-1}\right)\cdot 1.  
\end{split}
\end{equation}
\end{proof}

\begin{remark}
It is possible to include the class $\kappa_0$ into the obtained cut-and-join description. For this purpose it is enough to use the elements of the Virasoro group of the form $\exp({\sum_{k\geq0} a_k \widehat{L}_k })$. We will not consider this deformation here.
\end{remark}

Algebraic topological recursion \eqref{pp} immediately follows from the $\hbar$ expansion of \eqref{taum}. In the polynomials  $\tau_{\alpha}^{(p)}({\bf t},{\bf s}) $ one can further separate the contributions with a given numbers of marked points and specify the recursion for these contributions.

\subsection{Coefficients of the cut-and-join operators}
Let us find other expressions for the cut-and-join operators ${\widehat{W}_\alpha}({\bf s})$,  which can be more convenient. Consider the {\em odd bosonic current}
\begin{equation}\label{bc}
\begin{split}
\widehat{J}(z)&=\sum_{k=0}^\infty\left((2k+1) t_{2k+1}z^{2k}+\frac{1}{z^{2k+2}}\frac{\p}{\p t_{2k+1}}\right)\\
&=\sum_{k\in{\mathbb Z}_{\odd}} \frac{\widehat{J}_k}{z^{k+1}}.
\end{split}
\end{equation}
If the Virasoro group operator $\widehat{V}$ associated to a series $f(z)$ by \eqref{V} contains only the even components $\widehat{L}_{2k}$, we have
\be
\widehat{J}^{\bf a}(z)=\widehat{V} \widehat{J}(z) \widehat{V}^{-1}= h'(z) \widehat{J}(h(z)),
\ee
where $h(z)$ is a series, inverse to $f(z)$, see \eqref{hfun}.
Here
\be
\widehat{J}^{\bf a}(z)=\sum_{k\in {\mathbb Z}}\frac{\widehat{J}_k^{\bf a}}{z^{k+1}},
\ee
where the operators $\widehat{J}_k^{\bf a}$ are given by \eqref{LJa}.

The cut-and-join operators \eqref{CAJ} are given by the residues
\be
{\widehat{W}_\alpha}=\frac{1}{2(2\alpha+1)}\res{\left(\frac{1}{z^{2\alpha-1}} \widehat{J}(z)_+\normordboson\widehat{J}(z)^2\normordboson+\frac{1}{4z^{2\alpha+1}}\widehat{J}(z)\right)},
\ee
where the normal ordering $\normordboson\dots\normordboson$ puts all $\widehat{J}_k$ with positive $k$ to the right of all $\widehat{J}_k$ with negative $k$, and for any formal series $g(z)=\sum_{k\in {\mathbb Z}}g_k z^k$ we define $\res{g(z)}=g_{-1}$ and $g(z)_+=\sum_{k\geq 0}g_k z^k$.

Let us consider the Virasoro algebra associated to the odd current $\widehat{J}(z)$,
\be
\widehat{L}^o(z)=:\sum_{k\in {\mathbb Z}_{\even}} \frac{\widehat{L}^o_k}{z^{k+2}},
\ee
where
\be
\widehat{L}^o(z)=\frac{1}{2} \normordboson \widehat{J}(z)^2\normordboson.
\ee
Operators $\widehat{L}^o_k$ satisfy the commutation relations of the Virasoro algebra.
Then the cut-and-join operators $\widehat{W}_\alpha$ are given by
\begin{equation}\label{pli}
\begin{split}
{\widehat{W}_\alpha}&=\frac{1}{2\alpha+1}\res{\left(\frac{1}{z^{2\alpha-1}} \widehat{J}(z)_+\widehat{L}^o(z)+\frac{1}{8z^{2\alpha+1}}\widehat{J}(z)\right)}\\
&=\frac{1}{2\alpha+1} \sum_{k=0}^\infty (2k+1) t_{2k+1}\left(\widehat{L}^o_{2k-2\alpha}+\frac{\delta_{k,\alpha}}{8}\right).
\end{split}
\end{equation}

Consider the Virasoro group operators $\widehat{V}_{\alpha}({\bf s})$ given by \eqref{Vopp}. They contain only even components $\widehat{L}_{2k}$. Then we have
\be
\widehat{V}_{\alpha}({\bf s}) \widehat{L}^o(z) \widehat{V}_{\alpha}({\bf s})^{-1}=h_\alpha'(z)^2  \widehat{L}^o(h_\alpha(z))+c_\alpha(z),
\ee
where $c_\alpha(z)\in {\mathbb Q}[{\bf s}][\![z]\!]$ are certain formal power series in $z$. Only $c_1(0)=s_1/30$ contributes to the cut-and-join operators ${\widehat{W}_\alpha}({\bf s})$.
Therefore, the cut-and-join operators ${\widehat{W}_\alpha({\bf s}})$ are given by
\begin{multline}
{\widehat{W}_\alpha({\bf s}})=\frac{1}{2\alpha+1}\res \left(
\frac{\left(h_\alpha'(z)\widehat{J}(h_\alpha(z))\right)_+}{z^{2\alpha-1}} \left(h_\alpha'(z)^2\widehat{L}^o(h_\alpha(z))+\frac{s_1}{30}\right)\right.\\
\left.+\frac{h_\alpha'(z)\widehat{J}(h_\alpha(z))}{8z^{2\alpha+1}}\right).
\end{multline}
We can rewrite them as
\be\label{opop}
\widehat{W}_\alpha({\bf s})=\frac{1}{2\alpha+1}\left(\sum_{m=0}^\infty  \sum_{k=0}^\infty A_{k,m}^\alpha({\bf s}) \widehat{J}_{2m-2k-1}\widehat{L}^o_{2k-2\alpha}+\sum_{k=-\alpha-1}^\infty C_k^\alpha({\bf s}) \widehat{J}_{2k+1}\right),
\ee
where the coefficients are given by the residues
\be
A_{k,m}^\alpha({\bf s})=\res \frac{1}{z^{2\alpha-1}} \left(h_\alpha'(z)h_\alpha(z)^{2k-2m}\right)_+ h_\alpha'(z)^2 h_\alpha(z)^{2\alpha-2k-2}
\ee
and
\be\label{Bmm}
C_k^\alpha({\bf s})=\frac{1}{8}\res \frac{1}{f_\alpha(z)^{2\alpha+1}z^{2k+2}}+\frac{s_1 \delta_{\alpha,1}}{30} \res  \frac{1}{f_1(z)z^{2k+2}}.
\ee
For ${\bf s}={\bf 0}$ we have  $A_{m,k}^\alpha({\bf 0})=\delta_{m,0}$ and $C_k^\alpha({\bf 0})=\frac{1}{8}\delta_{k,-\alpha-1}$, therefore, operators \eqref{opop} reduce to \eqref{pli}.
Polynomials $A_{k,m}^\alpha({\bf s})$ and $C_k^\alpha({\bf s})$ are homogeneous in ${\bf s}$ of degree $m$ and $k+\alpha+1$ respectively if we put $\deg s_k=k$.

For $k\geq m$ the coefficients $A_{k,m}^\alpha({\bf s})$ do not depend on $k$,
\begin{equation}
\begin{split}\label{Amm}
A_{k,m}^\alpha({\bf s})&=\res \frac{1}{z^{2\alpha-1}}  h_\alpha'(z)^3 h_\alpha(z)^{2\alpha-2m-2}\\
&=\res \frac{z^{2\alpha-2m-2}}{f_\alpha'(z)^2f_\alpha(z)^{2\alpha-1}}\\
&=:A_{m}^\alpha({\bf s})
\end{split}
\end{equation}
First coefficients $A_{m}^\alpha$, $A_{k,m}^\alpha$, and $C_k^\alpha$ are presented in Appendix \ref{AAA}.

Expression \eqref{opop} for the cut-and-join operators is rather convenient for computations. 
On the level $p$ of the recursion \eqref{pp} only terms with $k\leq\frac{(2\alpha+1)(p-1)}{2}+\alpha$ and $m\leq\frac{(2\alpha+1)(p-1)+1}{2}+\alpha$ contribute to the first summation and the terms with $k\leq \frac{(2\alpha+1)(p-1)-1}{2}$ contribute to the second summation. 
\begin{proposition}
For $p>0$ we have
\begin{multline}
\tau_{\alpha}^{(p)}({\bf t},{\bf s})=\frac{1}{p(2\alpha+1)}\left(
\sum_{m=0}^{\lfloor \frac{(2\alpha+1)(p-1)+1}{2}+\alpha\rfloor}  \sum_{k=0}^{\lfloor\frac{(2\alpha+1)(p-1)}{2}+\alpha\rfloor} A_{k,m}^\alpha({\bf s}) \widehat{J}_{2m-2k-1}\widehat{L}^o_{2k-2\alpha}\right.\\
\left.+\sum_{k=-\alpha-1}^{\lfloor \frac{(2\alpha+1)(p-1)-1}{2}\rfloor} C_k^\alpha({\bf s}) \widehat{J}_{2k+1}\right)\cdot  \tau_{\alpha}^{(p-1)}({\bf t},{\bf s}).
\end{multline}
\end{proposition}
Moreover, only a finite number of the terms in the Virasoro operators $\widehat{L}^o_{2k-2\alpha}$ contribute, therefore, the algebraic topological recursion \eqref{pp} is given by polynomial differential operators acting on polynomials.

\subsection{Weil--Petersson volumes}
Let us consider the case with insertion of only one $\kappa$ class, namely $\kappa_1$. It describes the Weil--Petersson volumes of the moduli spaces of (super) hyperbolic Riemann surfaces with geodesic boundaries, see \cite{NorbS} for more details.

These volumes for the moduli spaces of hyperbolic Riemann surfaces (for $\alpha=1$) and for  the moduli spaces of super hyperbolic Riemann surfaces (for $\alpha=0$) are given by the integrals
\be
V_{g,n}^{\alpha}(L_1,\dots,L_n)=\int_{\overline{\mathcal{M}}_{g,n}} \Theta_{g,n}^{1-\alpha}\exp\left(2\pi^2 \kappa_1+\frac{1}{2}\sum_{j=1}^n L_j^2 \psi_j\right).
\ee
Here $L_1,\dots,L_n$ are the lengths of  the geodesic boundary components. These volumes are given \cite{NorbS} by the specifications of polynomials (\ref{FE})
\be
V_{g,n}^{\alpha}(L_1,\dots,L_n)=n! F_{g,n}^\alpha({\bf T}, {\bf s})\Big|_{s_k=2\pi^2 \delta_{k,1},\, 
T_k=\frac{1}{2^k k!}\sum_{j=1}^n L_j^{2k}}.
\ee
This case corresponds to $s_k=s \delta_{k,1}$, and in Theorem \ref{TMZN} corresponding generating functions are
described by the following shifts of the variables ${\bf t}$,
\be
t_{2k+1}\mapsto t_{2k+1}-\frac{1}{\hbar}\frac{(-s)^{k-\alpha}}{(2k+1)!!(k-\alpha)!}.
\ee
Let us introduce
\be
q=\sqrt{-2s}.
\ee
Then we have
\begin{equation}
\begin{split}
{v}^\alpha(z)&=\frac{z^{2\alpha+1}}{2\alpha+1}-\frac{1}{2i \sqrt{2\pi}}\int_\gamma\frac{{\hbox{d}} t}{t^{\alpha+\frac{3}{2}}}\exp\left(\frac{z^2 t}{2}-\frac{s}{t}\right)\\
&=\frac{z^{2\alpha+1}}{2\alpha+1}-\sqrt{\frac{\pi}{2}}\left(\frac{z}{q}\right)^{\alpha+\frac{1}{2}}I_{\alpha+1/2}(qz),
\end{split}
\end{equation}
where $I_{\alpha+1/2}$ is the modified Bessel function. Therefore, the functions $f_\alpha(z)$, which describe the Virasoro group elements in Proposition \ref{propV}, satisfy
\be
\frac{f_\alpha(z)^{2\alpha+1}}{2\alpha+1}=\sum_{k=0}^\infty\frac{(-s)^k}{k!}\frac{z^{2k+2\alpha+1}}{(2k+2\alpha+1)!!}.
\ee

For $\alpha=0$ we have
\begin{equation}
\begin{split}\label{fs}
f_0(z)&=\sum_{k=0}^\infty\frac{(-s)^k}{k!}\frac{z^{2k+1}}{(2k+1)!!}\\
&=\frac{1}{q}\sinh(qz).
\end{split}
\end{equation}
In this case it is possible to find all coefficients $C_m^0$ and $A_m^0$, given by \eqref{Bmm} and \eqref{Amm}, explicitly in terms of the Euler numbers and Bernoulli polynomials. Recall that the Bernoulli polynomials $B_k(x)$ are given by the generating function
\be
\frac{ze^{xz}}{e^z-1}=\sum_{k=0}^\infty B_k(x)\frac{z^k}{k!}.
\ee
For the Euler numbers $E_k$ we have
\be
\frac{2}{e^z-e^{-z}}=\sum_{k=0}^\infty E_k\frac{z^k}{k!}.
\ee
\begin{proposition}
\begin{equation}
\begin{split}
A_m^0({\bf s})\Big|_{s_k=s \delta_{k,1}}&=-\frac{q^{2m}}{(2m+1)!}E_{2m+2},\\
C_m^0({\bf s})\Big|_{s_k=s \delta_{k,1}}&=\frac{(2q)^{2m+2}}{8(2m+2)!}  B_{2m+2}(1/2).
\end{split}
\end{equation}
\end{proposition}
\begin{proof}
From \eqref{Amm} and \eqref{Bmm} it follows that $A_m^0$ and $C_m^0$ are the coefficients of the series expansion of $\frac{f_0(z)}{f'_0(z)^2}$ and $\frac{1}{8f_0(z)}$ respectively. Then for \eqref{fs} we have
\begin{equation}
\begin{split}
\frac{f_0(z)}{f'_0(z)^2}&=-\frac{1}{q^2}\frac{\p}{\p z} \frac{1}{f'_0(z)}=-\sum_{k=0}^\infty \frac{E_{k+1}z^k q^{k-1}}{k!},\\
\frac{1}{8f_0(z)}&=\frac{1}{8}\frac{2q e^{qz}}{e^{2qz}-1}=\frac{1}{8}\sum_{k=0}^\infty \frac{B_{k}(1/2) (2qz)^k}{k!}.
\end{split}
\end{equation}
\end{proof}
It is also possible to find the coefficients $A^0_{k,m}$ for $k<m$, for instance
\be
A^0_{m-1,m}=\frac{q^{2m}}{(2m)!}\left(2^{2m}B_{2m}(1/2)-\frac{E_{2m+2}}{2m+1}\right).
\ee
All higher coefficients can also be expressed in terms of the Euler numbers and Bernoulli polynomials.

For $\alpha=1$ we have
\begin{equation}
\begin{split}
f_1(z)&=\left(3\frac{z q\cosh(qz)-\sinh(qz)}{q^3}\right)^\frac{1}{3}\\
&=z+\frac{1}{30}\,{q}^{2}{z}^{3}+{\frac {1}{12600}}\,{q}^{4}{z}^{5}+{\frac {1}{
226800}}\,{q}^{6}{z}^{7}-{\frac {289}{2095632000}}\,{q}^{8}{z}^{9}+O \left( {z}^{11}  \right).
\end{split}
\end{equation}
In this case we have
\begin{proposition}
\be
A_m^1({\bf s})\Big|_{s_k=s \delta_{k,1}}=-12\frac{(2q)^{2m}}{(2m+1)!}B_{2m+2}(1/2).
\ee 
\end{proposition}

\appendix
\section{Coefficients of the cut-and-join operators}\label{AAA}

\begin{flalign}
  \begin{aligned}
A_{{1}}^0&=\frac{5}{3}\,s_{{1}},\\
A_{{2}}^0&={\frac {61}{30}}\,{s_{{1}}}^{2}+\frac{3}{5}\,s_{{2}},\\
A_{{3}}^0&={\frac {277}{126}}\,{s_{{1}}}^{3}+{\frac {479}{315}}\,s_{{2}
}s_{{1}}+{\frac {13}{105}}\,s_{{3}},\\
A_{{4}}^0&={\frac {43}{135}}\,s_{{3}}s_{{1}}+{\frac {50521}{22680}}\,{s
_{{1}}}^{4}+{\frac {4757}{1890}}\,s_{{2}}{s_{{1}}}^{2}+{\frac {17}{945
}}\,s_{{4}}+{\frac {529}{1890}}\,{s_{{2}}}^{2},\\
A_{{5}}^0&={\frac {1}{495}}\,s_{{5}}+{\frac {97}{2079}}\,s_{{4}}s_{{1}}
+{\frac {2011}{17325}}\,s_{{3}}s_{{2}}+{\frac {7891}{14850}}\,s_{{3}}{
s_{{1}}}^{2}+{\frac {98489}{103950}}\,{s_{{2}}}^{2}s_{{1}}+{\frac {
357839}{103950}}\,s_{{2}}{s_{{1}}}^{3}\\
&+{\frac {540553}{249480}}\,{s_{{
1}}}^{5}.
\end{aligned}&&&
\end{flalign}

\begin{flalign}
  \begin{aligned}
A_{0,1}^0({\bf s})&=\frac{1}{3}s_1,\\
A_{0,2}^0({\bf s})&={\frac {7}{90}}\,{s_{{1}}}^{2}+\frac{1}{15}\,s_{{2}},\\
A_{1,2}^0({\bf s})&={\frac {71}{90}}\,{s_{{1}}}^{2}+\frac{1}{5}\,s_{{2}},\\
A_{0,3}^0({\bf s})&={\frac {31}{1890}}\,{s_{{1}}}^{3}+{\frac {11}{315}}\,s_{{2}}s_{{1}}+{\frac {1}{105}}\,s_{{3}},\\
A_{1,3}^0({\bf s})&={\frac {37}{126}}\,{s_{{1}}}^{3}+{\frac {89}{315}}\,s_{{2}}s_{{1}}+\frac{1}{35}\,s_{{3}},\\
A_{2,3}^0({\bf s})&={\frac {2171}{1890}}\,{s_{{1}}}^{3}+{\frac {223}{315}}\,s_{{2}}s_{{1}}+\frac{1}{21}\,s_{{3}}.
\end{aligned}&&&
\end{flalign}

\begin{flalign}
  \begin{aligned}
 C_0^0&=\frac{1}{24}\,s_{{1}},\\
 C_1^0&={\frac {7}{720}}\,{s_{{1}}}^{2}+{\frac {1}{120}}\,s_{{2}},\\
 C_2^0&={\frac {31}{15120}}\,{s_{{1}}}^{3}+{\frac {11}{2520}}\,s_{{2}}s_{{1}}+{\frac {1}{840}}\,s_{{3}},\\
C_3^0&={\frac {127}{302400}}\,{s_{{1}}}^{4}+{\frac {113}{75600}}\,s_{
{2}}{s_{{1}}}^{2}+{\frac {1}{1512}}\,s_{{3}}s_{{1}}+{\frac {37}{75600}}\,{s_{{2}}}^{2}+{\frac {1}{7560}}\,s_{{4}},\\
C_4^0&={\frac {19}{249480}}\,s_{{4}}s_{{1}}+{\frac {61}{415800}}\,s_{
{3}}s_{{2}}+{\frac {587}{2494800}}\,s_{{3}}{s_{{1}}}^{2}+{\frac {179}{
498960}}\,{s_{{2}}}^{2}s_{{1}}+{\frac {1073}{2494800}}\,s_{{2}}{s_{{1}
}}^{3}\\
&+{\frac {1}{83160}}\,s_{{5}}+{\frac {73}{855360}}\,{s_{{1}}}^{5}.
\end{aligned}&&&
\end{flalign}

\begin{flalign}
  \begin{aligned}
A_{{1}}^1&={\frac {7}{15}}\,s_{{1}},\\
A_{{2}}^1&={\frac {31}{210}}\,{s_{{1}}}^{2}+{\frac {11}{105}}\,s_{{2}},\\
A_{{3}}^1&={\frac {127}{3150}}\,{s_{{1}}}^{3}+{\frac {113}{1575}}\,s_{{
2}}s_{{1}}+{\frac {1}{63}}\,s_{{3}},\\
A_{{4}}^1&={\frac {587}{51975}}\,s_{{3}}s_{{1}}+{\frac {73}{7128}}\,{s_
{{1}}}^{4}+{\frac {1073}{34650}}\,s_{{2}}{s_{{1}}}^{2}+{\frac {19}{
10395}}\,s_{{4}}+{\frac {179}{20790}}\,{s_{{2}}}^{2},\\
A_{{5}}^1&={\frac {23}{135135}}\,s_{{5}}+{\frac {899}{675675}}\,s_{{4}}
s_{{1}}+{\frac {12637}{4729725}}\,s_{{3}}s_{{2}}+{\frac {5249}{1051050
}}\,s_{{3}}{s_{{1}}}^{2}+{\frac {40961}{5255250}}\,{s_{{2}}}^{2}s_{{1}
}\\
&+{\frac {1540453}{141891750}}\,s_{{2}}{s_{{1}}}^{3}
+{\frac {1414477}{
567567000}}\,{s_{{1}}}^{5}.
\end{aligned}&&&
\end{flalign}

\begin{flalign}
  \begin{aligned}
A_{0,1}^1({\bf s})&=\frac{1}{15}\,s_{{1}},\\
A_{0,2}^1({\bf s})&={\frac {13}{3150}}\,{s_{{1}}}^{2}+{\frac {1}{105}}\,s_{{2}},\\
A_{1,2}^1({\bf s})&={\frac {137}{3150}}\,{s_{{1}}}^{2}+\frac{1}{35}\,s_{{2}},\\
A_{0,3}^1({\bf s})&={\frac {41}{141750}}\,{s_{{1}}}^{3}+{\frac {1}{675}}\,s_{{2}}s_{{1}}+{
\frac {1}{945}}\,s_{{3}},\\
A_{1,3}^1({\bf s})&={\frac {37}{9450}}\,{s_{{1}}}^{3}+{\frac {17}{1575}}\,s_{{2}}s_{{1}}+{
\frac {1}{315}}\,s_{{3}},\\
A_{2,3}^1({\bf s})&={\frac {2419}{141750}}\,{s_{{1}}}^{3}+{\frac {131}{4725}}\,s_{{2}}s_{{
1}}+{\frac {1}{189}}\,s_{{3}}.\\
\end{aligned}&&&
\end{flalign}

\begin{flalign}
  \begin{aligned}
   C_{-1}^1&={\frac {7}{120}}\,s_{{1}},\\
   C_{0}^1&={\frac {137}{25200}}\,{s_{{1}}}^{2}+{\frac {1}{280}}\,s_{{2}},\\
   C_{1}^1&={\frac {37}{75600}}\,{s_{{1}}}^{3}+{\frac {17}{12600}}\,s_{{2}}s_{{1}}+{\frac {1}{2520}}\,s_{{3}},\\
   C_{2}^1&={\frac {197}{1247400}}\,s_{{3}}s_{{1}}+{\frac {20539}{
87318000}}\,s_{{2}}{s_{{1}}}^{2}+{\frac {1}{27720}}\,s_{{4}}+{\frac {
163}{1940400}}\,{s_{{2}}}^{2}+{\frac {240043}{5239080000}}\,{s_{{1}}}^{4},\\
 C_{3}^1&=
{\frac {241}{16216200}}\,s_{{4}}s_{{1}}+{\frac {251}{
12612600}}\,s_{{3}}s_{{2}}+{\frac {100249}{3405402000}}\,s_{{3}}{s_{{1
}}}^{2}+{\frac {42187}{1135134000}}\,{s_{{2}}}^{2}s_{{1}}\\
&+{\frac {
113111}{3405402000}}\,s_{{2}}{s_{{1}}}^{3}+{\frac {1}{360360}}\,s_{{5}
}+{\frac {4477789\,{s_{{1}}}^{5}}{1021620600000}}.
\end{aligned}&&&
\end{flalign}

\bibliographystyle{alphaurl}
\bibliography{Weilref}

\end{document}